\documentclass[11pt]{amsart}
\usepackage[
    margin=1in
]{geometry}

\usepackage{amsmath,amssymb,amsxtra,amsthm,hyperref}
\usepackage{tikz}
\usepackage{multicol}
\usetikzlibrary{calc, decorations.markings}
\usepackage{enumitem} 

\newcommand{\condref}[1] {\hyperref[cond.#1]{(#1)}}

\newcommand{\cB}{\mathcal{B}}

\newcommand{\cH}{\mathcal{H}}

\newcommand{\bF}{\mathbb{F}}

\newtheorem{theorem}{Theorem}[section]

\theoremstyle{definition}
\newtheorem{definition}[theorem]{Definition}

\newtheorem{theoremx}{Theorem}

\numberwithin{equation}{section}

\author{Adam Dor-On}
\address{Department of Mathematics, University of Haifa, Mount Carmel, Haifa 3103301, Israel}
\email{adoron.math@gmail.com\vspace{-1ex}}

\author{Boyu Li}
\address{Department of Mathematical Sciences, New Mexico State University, Las Cruces, New Mexico, 88003, USA}
\email{boyuli@nmsu.edu}



\thanks{A. Dor-On was partially supported by an NSF / BSF grant DMS-2452324 / 2024734 (respectively), and a DFG Middle-Eastern collaboration project no. 529300231. A. Dor-On and B. Li were partially supported by an NSF-BSF grant DMS-2350543 / 2023695 (respectively).}

\title[A Nica-amenable monoid which fails to embed in an amenable group]{A Nica-amenable submonoid of a group which fails \\ to embed in an amenable group}

\begin{document}

\begin{abstract}
We construct a weakly quasilattice-ordered submonoid of a group whose full semigroup $\mathrm{C}^*$-algebra is nuclear, while the monoid does not embed into any amenable group. Hence, Nica amenability, as well as the nuclearity of its semigroup C*-algebra, do not characterize embeddability into an amenable group for group-embeddable monoids.
\end{abstract}

\maketitle

\section{Introduction}

Amenability plays a central role in Mathematics, and especially in the study of operator algebras associated with various group-related structures. A classical result in the theory states that the amenability of a group $G$ is characterized either by its full group C*-algebra $\textrm{C}^*(G)$ being nuclear, or by the canonical map $\lambda:\textrm{C}^*(G)\to \textrm{C}^*_r(G)$ being injective (see, for example, \cite[Theorem 2.6.8]{BrownOzawaBook}). 
On the other hand, in the case of semigroups, precise characterizations of different notions of amenability of the semigroup in terms of their semigroup C*-algebras are far more nuanced. For instance, Murphy \cite{Murphy1996_CommIso} proved that the universal C*-algebra generated by isometric representations of $\mathbb{N}^2$ is not nuclear, even though one would expect $\mathbb{N}^2$, being a rather straightforward abelian semigroup, to be amenable in any appropriate sense and therefore have a nuclear semigroup C*-algebra. This intriguing result motivated the development of semigroup C*-algebras, where it was eventually understood that additional relations on semigroup representations must be imposed. 

In \cite{Nica1992}, Nica first introduced the notion of \emph{Nica-covariant} relations for isometric representations of quasi-lattice ordered semigroups, and defined the semigroup C*-algebra to be the universal C*-algebra generated by such representations. Under these Nica-covariance conditions, the semigroup C*-algebra $\textrm{C}^*(\mathbb{N}^2)$ is indeed nuclear and is isomorphic to the reduced semigroup C*-algebra $\textrm{C}^*_r(\mathbb{N}^2)$, which is the one generated by the left regular representation of $\mathbb{N}^2$. Building on Nica's work, the theory of semigroup C*-algebras was extensively developed by multiple authors over several decades \cite{LacaRaeburn1996, CrispLaca2002, CDL2013, XLi2013_nuc, Starling2015, ABCD2021, LS2022}, to list only some. 

Prompted by the equivalence of amenability in the case of groups, a semigroup $P$ is called Nica-amenable if the canonical map from the full semigroup C*-algebra $\textrm{C}^*(P)$ to the reduced semigroup C*-algebra $\textrm{C}^*_r(P)$ is an isomorphism. Nica-amenability is closely related to the nuclearity of the semigroup C*-algebra. It was first speculated in \cite{LacaRaeburn1996} that Nica-amenable semigroups have nuclear semigroup C*-algebras, and this question remains open to this day. On the other hand, it is known that the converse holds, namely, that nuclearity of the universal semigroup C*-algebra implies that $P$ is Nica-amenable \cite[Theorem 5.6.44]{CELY2017} (see also \cite[Theorem 4.9]{LS2022}). 

Thus, it is natural to try to find a characterization of Nica-amenability of semigroups. One surprising example of a Nica-amenable semigroup is the free semigroup $\bF_k^+$, which, despite its natural overgroup being non-amenable, has a nuclear semigroup C*-algebra $\textrm{C}^*(\mathbb{F}_k^+)$, and therefore $\bF_k^+$ is Nica-amenable. In fact, this phenomenon goes deeper than initially expected, since it turns out that the free monoid $\mathbb{F}_k^+$ actually embeds inside an amenable group (for example, $\mathbb{F}_k/\mathbb{F}_k''$  \cite{Hoc69}). In fact, in \cite[Theorem 5.6.44]{CELY2017} (see also \cite[Corollary 4.10]{LS2022}) it was shown that when a semigroup $P$ embeds inside an amenable group, then $\textrm{C}^*(P)$ is nuclear and $P$ is Nica-amenable. The question of whether there exists a submonoid of a group whose semigroup C*-algebra is nuclear but does not embed in an amenable group, first posed in \cite[Section 5.11]{CELY2017}, has been an interesting open question ever since. 

Many known classes of Nica-amenable semigroups can be shown to embed inside amenable groups. For a quasilattice-ordered semigroup $(G,P)$ with a semilattice, Crisp and Laca \cite{CrispLaca2002} used a technique of Nica \cite{Nica1992} to prove that $P$ being Nica-amenable implies that $G$ is an amenable group.  For the right LCM semigroup $P_T$ constructed from a graph of monoids, it is shown in \cite[Remark 7.6]{ChenLi2023} that the reduced semigroup C*-algebra $\textrm{C}^*_r(P_T)$ is nuclear if and only if $P_T$ can be embedded inside an amenable group. One potential class of examples comes from Artin monoids. The Nica-amenability of Artin monoids was fully characterized in \cite{LL2020}. That is, an Artin monoid is Nica-amenable if and only if it is right-angled. This led to a question posed by X. Li on whether right-angled Artin monoids can be embedded inside amenable groups. We show that all right-angled Artin monoids can actually be embedded inside amenable groups. This answers a question posed by X. Li during a CARMA and AMSI workshop in 2017. 

\begin{theoremx}
Let $A_{\Gamma}^+$ be a right-angled Artin monoid associated to a finite undirected graph $\Gamma$. Then $A_{\Gamma}^+$ can be embedded injectively as a submonoid of an amenable group.
\end{theoremx}

As a result, an Artin monoid can be embedded inside an amenable group if and only if it is Nica-amenable. The combined evidence seems to indicate that Nica-amenable submonoids of groups will embed in amenable groups. However, in this paper, we find a group-embeddable weakly quasilattice-ordered monoid $P$ with a nuclear semigroup C*-algebra which cannot be embedded inside an amenable group. 

\begin{theoremx}
There exists a weakly quasilattice-ordered submonoid $P$ of a group $G$ whose semigroup C*-algebra $\textrm{C}^*(P)$ is nuclear, while $P$ does not embed as a submonoid of an amenable group.
\end{theoremx}

As a result, Nica-amenability and the nuclearity of the semigroup C*-algebra cannot be characterized by embeddability into an amenable group. Our idea for constructing our example is inspired by two standard principles in the theory of group actions. The first is that an appropriate wreath product "enlarges" a boundary, even when the group is amenable. This goes back to the work of Kaimanovich and Vershik on boundary theory for random walks, where an amenable lamplight group was shown to have a non-trivial Poisson boundary \cite[Subsection 6.2]{KV1983_RandWalk}. The second principle for our example is the simple observation that amenability of an action by a discrete group on a space, which is equivalent to the coincidence of full and reduced crossed products (see \cite[Theorem 4.4.3]{BrownOzawaBook}), is not necessarily implemented via an amenable group (see for instance \cite[Theorem 5.3.15]{BrownOzawaBook}). This simple observation has led to various counterexamples in the literature, for instance to the recent counterexamples for the full Hao-Ng isomorphism problem \cite{DT2026}.

\vspace{6pt}

\textbf{Acknowledgments.} The first-named author is grateful to Ian Thompson for discussions on amenability of group actions, held at the University of Copenhagen in early July 2026.

\section{Preliminary}

A semigroup $P$ is a set with an associative multiplication. When $P$ has an identity element, we say that $P$ is a monoid, and we denote by $e\in P$ this identity element. Throughout this paper, we assume $P$ can be embedded inside a group $G$, so that $P$ is cancellative. The set of all invertible elements inside a monoid $P$ is denoted by $P^*$, which is also $P^*=P\cap P^{-1}\subset G$.

A monoid $P$ is called a right LCM monoid if for any $p,q\in P$ we have that $pP\cap qP$ is either $\emptyset$ or $rP$ for some $r\in P$. Notice that $r$ may not be unique when $P^*\neq\{e\}$ since $rP=sP$ for all $s\in rP^*$. When a right LCM monoid $P$ is contained in a group $G$ and satisfies $P^*=\{e\}$, we call the pair $(G,P)$ a weakly quasilattice ordered group \cite{ABCD2021}. Since the choice of $r$ is unique in a weakly quasilattice-ordered semigroup, we use the notation $r=p\vee q$. One can define a partial order on a weakly quasilattice-ordered group $(G,P)$, we may define a partial order on $G$ by setting $x\leq y$ whenever $x^{-1}y\in P$. Note that we often refer to $P$ as a weakly quasilattice-ordered semigroup without referencing the enveloping group $G$. 

Each semigroup $P$ is naturally associated with a left-regular representation on Hilbert space $\lambda: P\to\cB(\ell^2(P))$. More precisely, the Hilbert space $\ell^2(P)$ has the standard orthonormal basis $\{\delta_p\}_{p\in P}$, and $\lambda$ is defined by $\lambda(p)\delta_q=\delta_{pq}$. The reduced semigroup C*-algebra $C_r^*(P)$ is the C*-algebra generated by $\{\lambda(p): p\in P\}$. For the full semigroup C*-algebra of a right LCM monoid $P$, we say that an isometric representation $V:P\to\cB(\cH)$ is Nica-covariant if for all $p,q\in P$,
\[V(p)V(p)^*V(q)V(q)^* = \begin{cases}
    V(r)V(r)^*, &\text{ if } pP\cap qP=rP, \\
    0, &\text{ if }pP\cap qP=\emptyset.
\end{cases}\]
One may easily verify that $\lambda$ is an isometric Nica-covariant representation. The full semigroup C*-algebra $\textrm{C}^*(P)$ of a right LCM monoid $P$ is defined as the universal C*-algebra generated by isometric Nica-covariant representations. We say $P$ is Nica-amenable if $\lambda:\textrm{C}^*(P)\to C_r^*(P)$ is injective. Examples of Nica-amenable semigroups include abelian semigroups \cite{Nica1992}, right-angled Artin monoids \cite{CrispLaca2002}, and Baumslag-Solitar monoids \cite{ABCD2021}, whereas examples of non-Nica-amenable semigroups include non-right-angled Artin monoids \cite{CrispLaca2002, LL2020}.

A key technique in proving a semigroup is Nica-amenable is via the so-called ``controlled map''. The controlled map was first introduced in \cite{LacaRaeburn1996} to prove that the free product of abelian quasilattice-ordered semigroups is Nica-amenable. In \cite[Definition 3.6]{ABCD2021}, the notion of a controlled map is further relaxed. For the purpose of this paper, we adopt the following definition.

\begin{definition}\label{defn.control} Let $(G,P)$ and $(H,Q)$ be weakly quasilattice-ordered groups. Suppose that $\mu: G\to H$ is an order-preserving group homomorphism. We say $\mu$ is a controlled map if 
\begin{enumerate}
    \item For all $x,y\in P$ with $xP\cap yP\neq \emptyset$, we have $\mu(x)\vee \mu(y)=\mu(x\vee y)$. 
    \item For all $q\in Q$, there exists $\{s_\lambda\}\subset \mu^{-1}(q)\cap P$ such that $\mu^{-1}(q)\cap P$ can be written as a disjoint union 
    \[\mu^{-1}(q)\cap P = \bigsqcup_{\lambda} S_\lambda,\]
    where,
    \[S_\lambda =\{x\in \mu^{-1}(q)\cap P: x\geq s_\lambda\}.\]
\end{enumerate}
\end{definition}

We note that the collection $\{s_\lambda\}$ in the Definition~\ref{defn.control} must satisfy $s_\lambda P\cap s_\mu P=\emptyset$ for all $\mu\neq \nu$, since $\{S_\lambda\}$ are disjoint and $s_\lambda \in S_\lambda$. The original definition of a controlled map in \cite{LacaRaeburn1996} is a special case where $S_\lambda=\{s_\lambda\}$. The controlled map defined in \cite{ABCD2021} is more general than our version here, where each $S_\lambda$ is an increasing union of the form $\{x\in \mu^{-1}(q)\cap P: x\in s_{\lambda,n}\}$. To prove $\textrm{C}^*(P)$ is nuclear, it suffices to find a controlled map $\mu:P\to Q$ such that $Q$ is abelian and $\textrm{C}^*(\ker(\mu)\cap P)$ is nuclear \cite{LacaRaeburn1996, ABCD2021}.

A typical example of a weakly quasilattice-ordered semigroup is the free monoid $\bF_n^+$, generated by $n$ free generators $\{e_1,\dots, e_n\}$. For $x,y\in \bF_n^+$, $x\bF_n^+\cap y\bF_n^+\neq \emptyset$ when $x$ and $y$ are prefix-comparable, which means either $x=yy'$ or $y=xx'$ for some $x', y'\in\bF_n^+$. 


\section{Right-angled Artin monoids}

In this section, we prove that right-angled Artin monoids can be embedded inside amenable groups. Let us recall some basic facts about right-angled Artin monoids. Let $\Gamma=(V,E)$ be an undirected graph with vertices $V=\{1,2,\dots,n\}$ and edges $E$. The right-angled Artin monoid is defined as 
\[
A_\Gamma^+=\langle e_1,\dots, e_n: e_ie_j=e_je_i, \text{ for all } ij\in E\rangle.
\]
In other words, each vertex $i$ corresponds to a generator $e_i$ and the generators $e_i$ and $e_j$ commute if and only if $ij$ is an edge of $\Gamma$. Right-angled Artin monoids form an important class of right LCM monoids in the study of semigroup C*-algebras \cite{CrispLaca2002, CrispLaca2007, Li2017, LL2020, BLi2026}.

Each element $w\in A_\Gamma^+$ can be written as $w=w_1w_2\cdots w_n$ where each $w_i\in\{e_1,\dots, e_n\}$. Here, this expression is not unique since one can replace $w_iw_{i+1}$ by $w_{i+1}w_i$ whenever their corresponding vertices are the same or are adjacent. However, since the relations in $A_\Gamma^+$ are homogeneous, the total number of syllables used in $w$ is fixed, which is called the length of $w$, denoted by $\ell(w)$. 

Recall that a vertex $i$ is an initial vertex of $w$ if $w=e_i w'$. For an element $w=w_1\cdots w_k e_i w'$ where $w_t\neq e_i$ for all $1\leq t\leq k$, $i$ is an initial vertex if and only if each $w_t$ commutes with $e_i$ for all $1\leq t\leq k$. 

\begin{theorem} 
Let $A_{\Gamma}^+$ be a right-angled Artin monoid associated to a finite undirected graph $\Gamma$. Then $A_{\Gamma}^+$ can be embedded injectively as a submonoid of an amenable group.
\end{theorem}

\begin{proof} Suppose $\Gamma$ is defined on $n$ vertices. We first prove that $A_\Gamma^+$ can be embedded injectively inside the monoid $\mathbb{N}^n\times(\bF_2^{+})^m$, where $m=|\{(i,j): ij\notin E\}|$. 

For each $i$, define $\pi_i:A_\Gamma^+\to\mathbb{N}$ by first setting
\[\pi_i(e_k)=\begin{cases} 1, &\text{if } k=i, \\
0, &\text{otherwise},
\end{cases}\]
and extend $\pi_i$ to a monoid homomorphism. 

Similarly, for each $ij\notin E$, define $\pi_{i,j}:A_\Gamma^+\to\bF_2^+\cong\langle e_i, e_j\rangle$ by first setting
\[\pi_{i,j}(e_k)=\begin{cases} e_k, &\text{if } k=i \text{ or } j, \\
e, &\text{otherwise},
\end{cases}\]
and then extending $\pi_{i,j}$ to a monoid homomorphism. 

Define $\pi:A_\Gamma^+\to\mathbb{N}^n\times (\bF_2^{+})^m$ by
\[
\pi=\left(\prod_{i=1}^n \pi_i\right)\times\left(\prod_{ij\notin E} \pi_{i,j}\right).
\]
We prove that $\pi$ is injective. Let $x\in A_\Gamma^+$, we first observe that
\[\pi_i(e_k x)=\begin{cases} 1+\pi_i(x), &\text{if } k=i, \\
\pi_i(x), &\text{otherwise},
\end{cases}\quad 
\pi_{i,j}(e_kx)=\begin{cases} e_k\pi_{i,j}(x), &\text{if } k=i \text{ or } j, \\
\pi_{i,j}(x), &\text{otherwise}.
\end{cases}\]
Therefore, if $\pi(e_kx)=\pi(e_ky)$, then $\pi(x)=\pi(y)$. 

Suppose $x,y\in A_\Gamma^+$ have $\pi(x)=\pi(y)$. Observe that
\[\ell(x)=\sum_{i=1}^n \pi_i(x).\]
Therefore, $\ell(x)=\ell(y)$. We now prove by induction on $\ell(x)=\ell(y)$. It is trivial when $\ell(x)=\ell(y)=0$, in which case, $x=y=e$. 

Suppose $j$ is an initial vertex of $x$ so that $x=e_j x'$ for some $x'\in A_\Gamma^+$. Since $\pi_j(y)=\pi_j(x)=1+\pi_j(x')\geq 1$, we know $e_j$ is included somewhere in any expressions of $y$. We claim $j$ must be an initial vertex for $y$. If not, then we can write $y$ as 
\[y=y_1y_2\dots y_k e_j y_{k+2}\dots y_m,\]
where each $y_s\in\{e_1,\dots, e_n\}$, none of $\{y_s\}_{s=1}^k$ is $e_j$, and there exists $1\leq t\leq k$ such that $y_t=e_i$ for some $i$ where $ij\notin E$. Fix such an $i$, and notice that $\pi_{i,j}(y)$ starts with $e_i$ because $e_i$ comes before $e_j$. However, $x=e_je'$ and $\pi_{i,j}(x)=e_j\pi_{i,j}(x')$ starts with $e_j$, which contradicts $\pi_{i,j}(x)=\pi_{i,j}(y)$. Therefore, $j$ must also be an initial vertex for $y$, and $y=e_jy'$ for some $y'\in A_\Gamma^+$. 

Now $\pi(e_j x') = \pi(x)=\pi(y)=\pi(e_j y')$, and thus by an earlier observation we have that $\pi(x')=\pi(y')$. This finishes the induction, and thus $\pi$ is injective.

Now, since the free monoid $\bF^+_2$ embeds injectively inside the solveable group $\bF_2/\bF_2''$, we see that the monoid $\mathbb{N}^n\times (\bF_2^{+})^m$ also embeds inside the amenable group $\mathbb{Z}^n\times (\bF_2/\bF_2'')^m$. Composing this embedding with $\pi$, we get that $A_\Gamma^+$ embeds inside an amenable group. 
\end{proof}

\section{An Example}

We now construct an example of a weakly quasilattice-ordered semigroup with a nuclear full semigroup C*-algebra, which cannot be embedded inside any amenable group. 

Let $\mathbb{F}_n=\mathbb{F}(s_1,\ldots,s_n)$ for $n\geq 2$, and let $\mathbb{F}_n^+$ denote the free monoid generated by $s_1,\ldots,s_n$. Let $I$ be a countable set and suppose that $\rho\colon \mathbb{F}_n\longrightarrow \operatorname{Sym}(I)$ is a permutation representation such that $\rho(\mathbb{F}_n)$ is non-amenable. This can be obtained, for instance, by taking $I = \mathbb{F}_n$ and having $\mathbb{F}_n$ act by left multiplication. Consider the abelian group and its positive cone
$$
    A=\mathbb{Z}^{(I)}
    =
    \bigoplus_{i\in I}\mathbb{Z}\delta_i,
    \qquad
    A_+=\mathbb{N}^{(I)}
    =
    \bigoplus_{i\in I}\mathbb{N}\delta_i.
$$
The permutation representation $\rho$ induces an automorphic action $\alpha\colon \mathbb{F}_n\curvearrowright A$ determined by $\alpha_g(\delta_i)=\delta_{\rho(g)i}$ for $g\in\mathbb{F}_n,\ i\in I$. Notice also that $\alpha_g(A_+)=A_+$ for every $g\in\mathbb{F}_n$. We form the semidirect-product group $G=A\rtimes_\alpha\mathbb{F}_n$ and its semidirect-product submonoid $P=A_+\rtimes_\alpha\mathbb{F}_n^+$. Here, the multiplication is given by $(x,u)(y,v) =\bigl(x+\alpha_u(y),uv\bigr)$ for $x,y\in A$ and $u,v\in\mathbb{F}_n$.

\begin{theorem}
\label{thm:example}
Let $I$ and $\rho$ be as above, and let $P=\mathbb{N}^{(I)}\rtimes_\alpha\mathbb{F}_n^+$. Then $P$ is a countable, weakly quasilattice-ordered, group-embeddable semigroup so that the full semigroup $\mathrm{C}^*$-algebra $\textrm{C}^*(P)$ is nuclear (and therefore $P$ is Nica amenable) while there is no injective monoid homomorphism from $P$ into an amenable group.
\end{theorem}

\begin{proof}
Since $P$ is a submonoid of the group $G$, it is cancellative and
group-embeddable. We first verify that $(0,e)\in P$ is the only invertible element in $P$. Suppose that
$(x,u),(y,v)\in P$ satisfy
\[
    (x,u)(y,v)=(0,e).
\]
It follows that $uv=e$ in $\mathbb{F}_n^+$, so $u=v=e$. Hence $x+y=0$ with $x,y\in A_+$, which
forces $x=y=0$. Therefore, $P\cap P^{-1}=\{e\}$. 

We next show that $P$ is right LCM. Let $(x,u)\in P$. Since
$\alpha_u(A_+)=A_+$, we have that
$$ 
    (x,u)P = \left\{
        \bigl(x+\alpha_u(z),uw\bigr):
        z\in A_+,\ w\in\mathbb{F}_n^+
    \right\} \notag = (x+A_+)\times u\mathbb{F}_n^+.
$$
Let $p=(x,u)$ and $q=(y,v)$ be elements of $P$. 
If the words $u,v\in\bF_n^+$ are not prefix-comparable, then $u\mathbb{F}_n^+\cap v\mathbb{F}_n^+=\varnothing$, and consequently $pP\cap qP=\varnothing$. Suppose now, without loss of generality, that $v=uw$ for some $w\in\mathbb{F}_n^+$. Then $u\mathbb{F}_n^+\cap v\mathbb{F}_n^+ =v\mathbb{F}_n^+$. For $x,y\in A_+$, we have that $x\vee y$ is their coordinatewise maximum. We then have $(x+A_+)\cap(y+A_+) =(x\vee y)+A_+$. It now follows from the description of $(x,u)P$ above that
\begin{align*}
    pP\cap qP
    &=
    \bigl((x\vee y)+A_+\bigr)
        \times v\mathbb{F}_n^+\\
    &=
    (x\vee y,v)P.
\end{align*}
Thus, $P$ is a right LCM monoid. Since $P\cap P^{-1}=\{e\}$ we get that $(G,P)$ is a weakly quasi-lattice ordered group, and, whenever $p$ and $q$ have a common upper bound, we have that $p\vee q = (x\vee y,u\vee v)$, where $u\vee v$ is the longer of the two prefix-comparable words $u$ and $v$.

We now prove Nica amenability. Let $\ell\colon\mathbb{F}_n\longrightarrow\mathbb{Z}$ be the length homomorphism determined by $\ell(s_j)=1$ for $1\leq j\leq n$. In particular, $\ell(w)=|w|$ for $w\in\mathbb{F}_n^+$. Define $\mu\colon G\longrightarrow\mathbb{Z}$ by setting $\mu(x,g)=\ell(g)$. We claim that $\mu\colon G \longrightarrow \mathbb{Z}$ is a controlled map in the sense of Definition~\ref{defn.control}. It is clear that $\mu(P)\subseteq\mathbb{N}$, and hence $\mu$ is
order-preserving. Suppose that $p=(x,u)$ and $q=(y,v)$ have a common upper bound. Then $u$ and $v$ are prefix-comparable. If without loss of generality $v=uw$, then $p\vee q=(x\vee y,v)$, and hence
$$
    \mu(p\vee q)
    =
    |v|
    =
    \max\{|u|,|v|\}
    =
    \mu(p)\vee\mu(q).
$$

For $k\in\mathbb{N}$ and for each $w\in \bF_n^+$ with $|w|=k$ take $s_w=(0,w)$. Notice that for all $(x,w)\in\mu^{-1}(k)\cap P$ we have $|w|=\mu(x,w)=k$ and $(x,w)=(0,w)\cdot(\alpha_w^{-1}(x),e)\geq s_w$. Therefore, 
\[\mu^{-1}(k)\cap P=\bigsqcup_{|w|=k} \{x\in \mu^{-1}(k)\cap P: x\geq s_w\}.\]

On the other hand, when $p$ and $q$ do not have a common upper bound, this occurs precisely when $v$ and $u$ are not prefix comparable. Consequently, by the description of $(x,u)P$, it follows that $p P\cap qP \subseteq A_+ \times (v\mathbb{F}_n^+ \cap u \mathbb{F}_n^+) = \varnothing$. This proves that $\mu$ is controlled.

We have $\ker\mu\cap P=A_+\times\{e\}$, so that the subgroup of $G$ generated by $\ker\mu\cap P$ is $\langle\ker\mu\cap P\rangle  = A\times\{e\}$, which is abelian and hence amenable. It follows from \cite[Corollary~4.4]{ABCD2021} that $\textrm{C}^*(\ker\mu\cap P)$ is nuclear. Since the target group $\mathbb{Z}$ is amenable,
\cite[Theorem~4.3]{ABCD2021} now implies that $\textrm{C}^*(P)$ is nuclear and
that $P$ is Nica-amenable.

It remains to show that $P$ cannot embed into an amenable group.
Suppose, toward a contradiction, that $H$ is an amenable group and
that $\pi\colon P\longrightarrow H$ is an injective monoid homomorphism.

Let $\pi_0\colon A_+\longrightarrow H$ be given by $\pi_0(x)=\pi(x,e)$. Since $A_+$ is commutative, its image under $\pi_0$ is a commuting
family. Therefore, $\pi_0$ extends uniquely to a group homomorphism $\widetilde{\pi}_0\colon A\longrightarrow H$ given by $\widetilde{\pi}_0(x-y) =   \pi_0(x)\pi_0(y)^{-1}$ for $x,y\in A_+$. This map is well-defined. Indeed, if $x-y=x'-y'$, then $x+y'=x'+y$ inside $A_+$, and hence $\pi_0(x)\pi_0(y') = \pi_0(x')\pi_0(y)$. Since the elements in $\pi_0(A_+)$ commute, this gives $\pi_0(x)\pi_0(y)^{-1}  = \pi_0(x')\pi_0(y')^{-1}$. Moreover, $\widetilde{\pi}_0$ is injective. Indeed, if $\widetilde{\pi}_0(x-y)=e$, then $\pi_0(x)=\pi_0(y)$. The injectivity of $\pi$ gives $x=y$, and therefore $x-y=0$.

Denote $B=\widetilde{\pi}_0(A)\leq H$, and, for $1\leq j\leq n$, define $E_j=\pi(0,e_j)$. For every $x\in A_+$, the multiplication in $P$ gives $(0,e_j)(x,e) = \bigl(\alpha_{e_j}(x),e\bigr)(0,e_j)$. Applying $\pi$, we obtain $E_j\pi_0(x)E_j^{-1} = \pi_0\bigl(\alpha_{e_j}(x)\bigr)$. By passage to the group completion, it follows that $E_j\widetilde{\pi}_0(a)E_j^{-1} = \widetilde{\pi}_0\bigl(\alpha_{e_j}(a)\bigr)$ for all $a\in A$. Thus, every $E_j$ normalizes $B$.

Let $L=\langle B,E_1,\ldots,E_n\rangle\leq H$. Then $B$ is a normal abelian subgroup of $L$. Conjugation on $B$ induces a group homomorphism $\operatorname{Ad}_B\colon L\longrightarrow\operatorname{Aut}(B)$ given by $\operatorname{Ad}_B(h)(b)=hbh^{-1}$. Using the isomorphism $\widetilde{\pi}_0\colon A\longrightarrow B$, we may identify $\operatorname{Aut}(B)$ with $\operatorname{Aut}(A)$.
Under this identification, we have that $\operatorname{Ad}_B(E_j)=\alpha_{e_j}$ for all  $1\leq j\leq n$. Since $B$ is abelian, conjugation by every element of $B$ acts
trivially on $B$. It follows that
$$
    \operatorname{Ad}_B(L)
    =
    \langle\alpha_{e_1},\ldots,\alpha_{e_n}\rangle
    =
    \alpha(\mathbb{F}_n).
$$
The natural homomorphism $\operatorname{Sym}(I)\longrightarrow\operatorname{Aut}(A)$ given by $\sigma\longmapsto
    \bigl(\delta_i\mapsto\delta_{\sigma(i)}\bigr)$ is injective. Therefore, $\alpha(\mathbb{F}_n)\cong\rho(\mathbb{F}_n)$, and since $\rho$ is injective, we see that $\operatorname{Ad}_B(L)\cong \alpha(\mathbb{F}_n) \cong \rho(\mathbb{F}_n) \cong \mathbb{F}_n$ is not amenable.

On the other hand, $L$ is a subgroup of the amenable group $H$, and
is therefore amenable. Every homomorphic image of an amenable group
is amenable, so $\operatorname{Ad}_B(L)$ must be amenable. This is a
contradiction. We conclude that $P$ does not embed into any amenable
group.
\end{proof}


\def\lfhook#1{\setbox0=\hbox{#1}{\ooalign{\hidewidth
  \lower1.5ex\hbox{'}\hidewidth\crcr\unhbox0}}}

\end{document}